\documentclass[proceedings%
]{aofa}

\usepackage[utf8]{inputenc}
\usepackage{subfigure}

\def\P{\mathbb{P}}
\def\E{\mathbb{E}}
\def\beq*{\begin{eqnarray*}}
\def\eeq*{\end{eqnarray*}}

\def\be{\begin{equation}}
\def\ee{\end{equation}}

\newtheorem{theorem}{Theorem}

\newtheorem{corollary}[theorem]{Corollary}
\newtheorem{proposition}[theorem]{Proposition}

\newtheorem{remark}{Remark}

\usepackage[round]{natbib}

\author[P. Hitczenko and A. Lohss]{Pawe{\l} Hitczenko\addressmark{1}\thanks{Partially supported by
    a Simons Foundation grant \#208766}\and Amanda Lohss\addressmark{1}}

\title[Polynomial recurrences]{Probabilistic consequences of some  polynomial recurrences}

\address{\addressmark{1}Department of Mathematics, Drexel University, Philadelphia, PA 19104, USA}

\begin{document}
\maketitle
\begin{abstract}
\paragraph{Abstract.}
  In this paper, we consider sequences of polynomials that satisfy differential--difference recurrences. Our interest is motivated by the fact that polynomials satisfying such recurrences frequently appear as generating polynomials of integer valued random variables that are of interest in discrete mathematics. It is, therefore, of interest to understand the properties of such polynomials and  their probabilistic consequences. As an illustration   we  analyze  probabilistic properties of  tree--like tableaux, combinatorial objects that are connected to  asymmetric exclusion processes. In particular, we show that the number of diagonal boxes in symmetric tree--like tableaux is asymptotically normal and that the number of occupied corners in a random tree--like tableau is asymptotically Poisson. This extends earlier results of Aval, Boussicault,  Nadeau, and Laborde Zubieta, respectively.

 \end{abstract}


\section{Introduction and motivation}\label{sec:mot} 
In this paper we will consider a sequence of polynomials  \[P_n(x)=\sum_{k=0}^mp_{n,k}x^k,\quad n\ge0\]  
that satisfy a differential--difference recurrence of one of the following forms 
\begin{eqnarray}
\label{rec_prime} P^{'}_n(x)
&=&f_n(x)P_{n-1}(x)+g_n(x)P_{n-1}^{'}(x)\\
\nonumber\mbox{or}\qquad\qquad&&\\
\label{rec}P_n(x)
&=&f_n(x)P_{n-1}(x)+g_n(x)P_{n-1}^{'}(x)
\end{eqnarray}
for some sequences of  polynomials \begin{math}(f_n)\end{math}, \begin{math}(g_n)\end{math} and a given \begin{math}P_0(x)\end{math}.
 
As a motivation for our interest we give examples of recurrences of these types  that we encountered in recent literature. The first two examples appear in the context of tree--like tableaux introduced in \cite{ABN}. 
\begin{itemize}

\item[\textbf{(ABN)}]{} ~\cite{ABN}: 
\begin{eqnarray*}B_{n}(x)&=&nx(x+1)B_{n-1}(x)+x(1-x^2)B'_{n-1}(x),\\ B_{0}(x)&=&x.\end{eqnarray*} 
\item[\textbf{(LZ)\,\, }]{}~\cite{LZ}:
\begin{eqnarray*}
P^{'}_n(x)&=&nP_{n-1}(x)+2(1-x)P_{n-1}^{'}(x),\\ P_0(x)&=&1.\end{eqnarray*}
\end{itemize}
Laborde Zubieta also considered the following version 
\begin{eqnarray*}
Q^{'}_n(x)&=&2nxQ_{n-1}(x)+2(1-x^2)Q_{n-1}^{'}(x),\\ Q_0(x)&=&1,
\end{eqnarray*}
where \begin{math}Q_n(x)\end{math} is a polynomial of degree \begin{math}2n\end{math} whose odd--numbered coefficients vanish. But this recurrence can be reduced to \textbf{(LZ)} by considering \begin{math}Q_n(x)=P_n(x^2)\end{math}.

The following recurrence for fixed parameters $a$ and $b$ was considered in~\cite{HJ} (see Sections~2 and~4 there):
\begin{itemize}
\item[\textbf{(HJ)}]{}~\cite{HJ}: 
\begin{eqnarray*}P_{n,a,b}(x)&=&((n-1+b)x+a)P_{n-1,a,b}(x)+x(1-x)P'_{n-1,a,b}(x)\\ P_{0,a,b}(x)&=&1.\end{eqnarray*} 
\end{itemize}
This is a generaliztion of the classical Eulerian polynomials.  Specifically,  the choice of parameters \begin{math}a=1\end{math} and \begin{math}b=0\end{math} gives \begin{math}P_{n,1,0}=E_n(x)\end{math}, where
\[E_n(x)=\sum_{k=0}^n\left<n\atop k\right>x^k,\]
 and \begin{math}\left< n\atop k\right>\end{math} is the number of permutations of \begin{math}\{1,\dots,n\}\end{math} with exactly \begin{math}k\end{math} ascents. The  recurrence for the polynomials \begin{math}E_n(x)\end{math}  is:
 \[E_n(x)=((n-1)x+1)E_{n-1}(x)+x(1-x)E'_{n-1}(x).\]
A very similar recurrence played a role in \cite{DHH} although it appeared there only implicitly.
\begin{itemize}
\item[\textbf{(DHH)}]{}~\cite{DHH}:
\begin{eqnarray*}
V_n(x)&=&((2n-1)x+1)V_{n-1}(x)+2x(1-x)V'_{n-1}(x) \\ V_0(x)&=& 1.
\end{eqnarray*}
\end{itemize}
As one more example, the following recurrence was used in~\cite[Section~3]{AH} in connection with the analysis of a version of a card game called the memory game. 
\begin{itemize}
\item[\textbf{(AH)}]{}~\cite{AH}:
\begin{eqnarray*}
A_n(x)&=&(2n-1)A_{n-1}(x)+x(x-1)A_{n-1}^{'}(x),\\ A_0(x)&=&x.
\end{eqnarray*}

\end{itemize}
In the examples above the polynomials are generating polynomials of integer valued random variables and it is of interest to understand what bearing the form of a recurrence has on the probabilistic properties of these  random variables. This is, of course, not a new idea and in various forms has been studied for a long time (see, for example, many results and references in \cite{FS}). Still, we believe that there is more work to be done to better understand the probabilistic consequences of the above recurrences. 
\section{Tree--like tableaux}
Although we would like to keep the discussion at a general level, we will use particular objects, namely tree--like tableaux as a primary illustration. Therefore we briefly introduce the definition and their basic properties; we refer the reader to \cite{ABN, LZ,HL} for more information and details.

A \emph{Ferrers diagram} is a left--aligned finite set of cells arranged in rows and columns  with weakly decreasing number of cells in rows. Its \emph{half--perimeter}  is the number of rows plus the number of columns. 
The \emph{border edges} of a Ferrers diagram are the edges of the southeast border, and the number of border edges is equal to the half--perimeter. 
A \emph{tree--like tableaux} of size \begin{math}n\end{math} 
is a Ferrers diagrams of half-perimeter \begin{math}n+1\end{math}  with some cells (called pointed cells) filled with a point according to the following rules:
\begin{enumerate}
\item The cell in the first column and first row is always pointed (this point is known as the root point). \label{T1}
\item Every row and every column contains at least one pointed cell. \label{T2}
\item For every pointed cell, all the cells above are empty or all the cells to the left are empty. \label{T3}
\end{enumerate}

We will also  consider \emph{symmetric tree--like tableaux}, a subset of tree--like tableaux which are symmetric about their main diagonal (see \cite[Section 2.2]{ABN} for more details). 
As noticed in \cite{ABN}, the size of a symmetric tree--like tableaux must be odd.
It is known that there are \begin{math}n!\end{math}  tree--like tableaux of size \begin{math}n\end{math}  (see \cite[Corollary~8]{ABN}) and \begin{math} 2^nn!\end{math}  symmetric tree--like tableaux of size \begin{math}2n+1\end{math}  (see \cite[Corollary~8]{ABN}). 

\emph{Corners} of a  tree--like tableau (symmetric or not)  are the cells in which both  the right and bottom edges are border edges. \emph{Occupied corners} are corners that contain a point. Figure~\ref{pics} shows examples of tree--like tableaux.

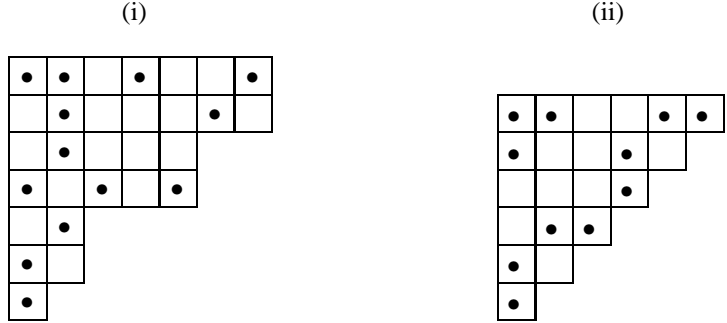
\begin{figure}
\setlength{\unitlength}{.5cm}
\begin{center}
\begin{picture} (22,8)

\put(3, 8){(i)}
\put(15.5, 8){(ii)}
\put(0,0){\line(0,1){7}}
\put(1,0){\line(0,1){7}}
\put(2,1){\line(0,1){6}}
\put(3,3){\line(0,1){4}}
\put(4,3){\line(0,1){4}}
\put(5,3){\line(0,1){4}}
\put(6,5){\line(0,1){2}}
\put(7,5){\line(0,1){2}}

\put(0,7){\line(1,0){7}}
\put(0,6){\line(1,0){7}}
\put(0,5){\line(1,0){7}}
\put(0,4){\line(1,0){5}}
\put(0,3){\line(1,0){5}}
\put(0,2){\line(1,0){2}}
\put(0,1){\line(1,0){2}}
\put(0,0){\line(1,0){1}}

\put(4.3, 3.3){$\bullet$}
\put(5.3, 5.3){$\bullet$}
\put(6.3, 6.3){$\bullet$}
\put(3.3, 6.3){$\bullet$}
\put(2.3, 3.3){$\bullet$}
\put(1.3, 2.3){$\bullet$}
\put(1.3, 4.3){$\bullet$}
\put(1.3, 5.3){$\bullet$}
\put(1.3, 6.3){$\bullet$}
\put(0.3, 3.3){$\bullet$}
\put(0.3, 1.3){$\bullet$}
\put(0.3, 0.3){$\bullet$}
\put(0.3, 6.3){$\bullet$}

\put(13,0){\line(0,1){6}}
\put(14,0){\line(0,1){6}}
\put(15,1){\line(0,1){5}}
\put(16,2){\line(0,1){4}}
\put(17,3){\line(0,1){3}}
\put(18,4){\line(0,1){2}}
\put(19,5){\line(0,1){1}}

\put(13,6){\line(1,0){6}}
\put(13,5){\line(1,0){6}}
\put(13,4){\line(1,0){5}}
\put(13,3){\line(1,0){4}}
\put(13,2){\line(1,0){3}}
\put(13,1){\line(1,0){2}}
\put(13,0){\line(1,0){1}}

\put(13.25, 0.25){$\bullet$}
\put(13.25, 1.25){$\bullet$}
\put(13.25, 4.25){$\bullet$}
\put(13.25, 5.25){$\bullet$}
\put(14.25, 2.25){$\bullet$}
\put(14.25, 5.25){$\bullet$}
\put(15.25, 2.25){$\bullet$}
\put(16.25, 3.25){$\bullet$}
\put(16.25, 4.25){$\bullet$}
\put(17.25, 5.25){$\bullet$}
\put(18.25, 5.25){$\bullet$}

\end{picture}
\caption{(i) A tree--like tableaux of size $13$   with $4$ corners and $2$ occupied corners. (ii) A  symmetric tree--like tableaux of size $11$ with $6$ corners, $4$ of which are occupied.\label{pics}}
\end{center}
\end{figure}

\section{General setting}
Motivated by  examples discussed in Section~\ref{sec:mot} we wish to consider a sequence of polynomials  \[P_n(x)=\sum_{k=0}^mp_{n,k}x^k,\quad n\ge0\]  
that satisfy one of the  recurrences (\ref{rec_prime}) or (\ref{rec})
 with the given initial polynomial \begin{math}P_0(x)\end{math}. The sequences of polynomials \begin{math}(f_n(x))\end{math}  and \begin{math}(g_n(x))\end{math}  are typically of low degree, but formally this is not required. Similarly, 
in all of the above examples we have  \begin{math}g_n(1)=0\end{math}  and we will assume that throughout. It should be emphasized, however, that there are natural situations in which the condition \begin{math}g_n(1)=0\end{math}  fails. For example,~\cite{W} considered a recurrence
\[T_n(x)=(x+c)T_{n-1}(x)+mxT^{'}_{n-1}(x),\]
for fixed numbers \begin{math}c\end{math}  and \begin{math}m\end{math}. The choice \begin{math}c=0\end{math}  and \begin{math}m=1\end{math}  is a classical situation of Bell polynomials (see e.~g. a discussion at the end of Section~7.2 in Chapter~VII of \cite{C}). Furthermore, the choice \begin{math}c=1\end{math}  and any fixed  \begin{math}m\in \Bbb N\end{math}  gives polynomials associated with Whitney numbers of Dowling lattices (see \cite{Be2}). For polynomials satisfying
\[F_n(x)=(x+1)F_{n-1}(x)+x(x+m)F_{n-1}'(x)\]
with \begin{math}m\in\Bbb N\end{math}  we refer to \cite[Section~4]{Be1} and references therein. So, clearly it is of interest to consider (\ref{rec_prime}) or (\ref{rec}) without the assumption that \begin{math}g_n(1)=0\end{math}  but as we indicated earlier we will assume this throughout this paper. 

Since we are interested in a probabilistic interpretation, we will assume that \begin{math}p_{n,k}\ge0\end{math}  and that 
\begin{math} \sum_kp_{n,k}>0\end{math}  for every \begin{math}n\end{math}. Then
\[\frac{P_n(x)}{P_n(1)}=\sum_{k\ge0}\frac{p_{n,k}}{P_n(1)}x^k\]
is the probability generating function of the integer valued random variable \begin{math}X_n\end{math}  whose distribution function is given  by 
\begin{equation}\label{rv}\P(X_n=k)=\frac{p_{n,k}}{P_n(1)},\quad k\ge0.\end{equation}

We note that  recurrence (\ref{rec_prime}) defines the polynomials \begin{math}P_n\end{math}  up to an additive constant or, equivalently, up to the value \begin{math}P_n(1)\end{math}. In our context the polynomials  arise in the study of discrete combinatorial structures, and thus a natural choice of  the normalization is obtained by letting \begin{math}P_n(1)\end{math}  be the cardinality of the structure consisting of all objects of size \begin{math}n\end{math}. For example, Laborde Zubieta set \begin{math}P_n(1)=n!\end{math}  and \begin{math}Q_n(1)=2^nn!\end{math}  representing the number of tree--like tableaux of size \begin{math}n\end{math}  and the symmetric tree--like tableaux of size \begin{math}2n+1\end{math} , respectively.   

We want to use recurrences (\ref{rec_prime}) and (\ref{rec}) to study the convergence in distribution of the sequences \begin{math}(X_n)\end{math}  associated with these recurrences through (\ref{rv}). 

\section{Method of moments}
One natural approach is to use the method of moments or, more precisely, the method of factorial moments. It is based on the fact that if \begin{math}X\end{math}  is a random variable uniquely determined by its (factorial) moments 
\[\E(X)_r=\E X(X-1)\dots(X-(r-1)), \quad r=1,2,\dots\]
and \begin{math}(X_n)\end{math}  is a sequence of random variables such that 
\[\E(X_n)_r\longrightarrow\E(X)_r,\quad n\to\infty,\quad r=1,2,\dots\]
then 
\[X_n\stackrel d\longrightarrow X,\quad n\to\infty,\]
where \lq\lq\begin{math}\stackrel d\longrightarrow\end{math} \rq\rq\ denotes the convergence in distribution.

As is well--known, for a random variable 
 \begin{math}X\end{math}  with probability generating function \begin{math}h(x)=\E x^X\end{math}  we have 
\[\E(X)_r=h^{(r)}(1),\]
where \begin{math}h^{(r)}(x)\end{math}  is the \begin{math}r^{th}\end{math}  derivative of \begin{math}h(x)\end{math}. 
Thus, in terms of polynomials \begin{math}(P_n(x))\end{math}  this means
\[\E(X_n)_r=\frac{P_n^{(r)}(1)}{P_n(1)}\]
and consequently, we would be interested in computing \begin{math}P_n^{(r)}(1)\end{math}  and finding the asymptotic of the ratio on the right--hand side above.

 For  recurrence  (\ref{rec_prime}) using Leibniz  formula for higher order derivative of the product we obtain
\begin{eqnarray*}
P_n^{(r)}(x)&=&\left(P_n'(x)\right)^{(r-1)}=\left(f_n(x)P_{n-1}(x)\right)^{(r-1)}+\left(g_n(x)P_{n-1}^{'}(x)\right)^{(r-1)}\\&=&
\sum_{k=0}^{r-1}{r-1\choose k}f_n^{(k)}(x)P_{n-1}^{(r-1-k)}(x)+
\sum_{k=0}^{r-1}{r-1\choose k}g_n^{(k)}(x)P_{n-1}^{(r-k)}(x)\\&=&
g_n(x)P_{n-1}^{(r)}(x)+\sum_{k=0}^{r-2}\left({r-1\choose k}f_n^{(k)}(x)+{r-1\choose k+1}g_n^{(k+1)}(x)\right)P_{n-1}^{(r-1-k)}(x)\\&&\qquad +f_n^{(r-1)}(x)P_{n-1}(x).
\end{eqnarray*}
The idea now is that  if \begin{math}f_n\end{math}  and \begin{math}g_n\end{math}  are low--degree polynomials  then one obtains a manageable  recurrence for \begin{math}P_n^{(r)}(1)\end{math}. We will illustrate this on Laborde Zubieta's example {\textbf(LZ)}. 
In that case \begin{math}f_n(x)\end{math}  and \begin{math}g_n(x)\end{math}  are  polynomials of degree zero and one, respectively and thus the above expression reduces to 
\be\label{red_diff}P_n^{(r)}(x)=g_n(x)P_{n-1}^{(r)}(x)+\left(f_n(x)+(r-1)g_n'(x)\right)P_{n-1}^{(r-1)}(x)
\ee
if \begin{math}r\ge2\end{math}  (and agrees with (\ref{rec_prime}) if \begin{math}r=1\end{math}).
 Laborde Zubieta used this, the specific form of the polynomials \begin{math}f_n(x)\end{math}, $g_n(x)$, and \begin{math}P_n(1)=n!\end{math}  to show that the random variables \begin{math}X_n\end{math}  defined by (\ref{rv})
 satisfy
 \[\E X_n=1\quad\mbox{and}\quad\textrm{var}(X_n)=\frac{n-2}n.\]   
This  suggests that the sequence \begin{math}(X_n)\end{math}  converges in distribution to a Poisson random variable with parameter 1. This is, indeed the case, and can be deduced from the recurrence (\ref{rec_prime}) as was shown in \cite{HL}.  Here is a general  statement that covers \textbf{(LZ)}.

\begin{proposition}\label{prop:pol_recur}
Let \[P_n(x)=\sum_{k=0}^mp_{n,k}x^k\] 
be a sequence of polynomials  satisfying  recurrence (\ref{rec_prime})
where \begin{math}f_n(x)=f_n\end{math}  and \begin{math}g_n(x)=g_n\cdot(x-1)\end{math}  for some sequences of constants \begin{math}(f_n)\end{math}  and \begin{math}(g_n)\end{math}. Assume that \begin{math}p_{n,k}\ge0\end{math}  and that  \begin{math}\sum_kp_{n,k}>0\end{math}  for every \begin{math}n\ge 1\end{math}, and that  \begin{math}m=m_n\end{math}   may depend on \begin{math}n\end{math}.  Consider a sequence of random variables \begin{math}(X_n)\end{math}  defined by (\ref{rv}).
If 
\begin{equation}\label{pol_ass}g_n=o(f_n)\quad\mbox{ and}\quad
f_n\frac{P_{n-1}(1)}{P_{n}(1)}\to c>0,\quad \mbox{as}\quad n\to\infty
\end{equation}
then 
\[X_n\stackrel{d}{\rightarrow}\textrm{Pois}(c)\quad \mbox{as} \quad n\to\infty,
\]
where \begin{math}\textrm{Pois}(c)\end{math}  is a Poisson random variable with parameter \begin{math}c\end{math}.
\end{proposition}
As established by~\cite{LZ}, the generating polynomials for the number of occupied corners in tree--like tableaux satisfy recurrence \textbf{(LZ)} (that means taking  \begin{math}f_n=n\end{math}, \begin{math}g_n=-2\end{math}, and  \begin{math}P_n(1)=n!\end{math}  in Proposition~\ref{prop:pol_recur}). 
Thus, the assumptions (\ref{pol_ass})  are clearly satisfied with \begin{math}c=1\end{math}  and we obtain the following extension of Laborde Zubieta's result (see \cite{HL})
\begin{corollary} 
As \begin{math}n\to\infty\end{math},  the limiting distribution of the number of occupied corners in a random tree--like tableau of size \begin{math}n\end{math}  is $\textrm{Pois}(1)$.
\end{corollary} 
A companion result for symmetric tableaux is as follows (see \cite{HL} for more details). The expected value and the variance were obtained earlier in~\cite{LZ}. 
\begin{corollary}
As \begin{math}n\to\infty\end{math},  the limiting distribution of the number of occupied corners in a random symmetric tree--like tableau of size \begin{math}2n+1\end{math}  is \begin{math}2\times\textrm{Pois}(1/2)\end{math}.
\end{corollary}

\begin{proof}[
of Proposition~\ref{prop:pol_recur}]
By \cite[Theorem~20, Chapter~1]{B} it is enough to  show that for every \begin{math}r\ge1\end{math}   the factorial moments 
\[\E(X_n)_r=\E X_n(X_n-1)\dots(X_n-(r-1)),\]
of \begin{math}(X_n)\end{math}  converge to \begin{math}c^r\end{math}  as \begin{math}n\to\infty\end{math}.       
Using \begin{math}g_n(1)=0\end{math}  and \begin{math}g_n'(x)=g_n\end{math}  in (\ref{red_diff}) we obtain 
\beq*
P_n^{(r)}(1)&=&\left(f_n+(r-1)g_n\right)P_{n-1}^{(r-1)}(1).
\eeq*
Consequently, 
\beq*
\frac{P_n^{(r)}(1)}{P_n(1)}&=&(f_n+(r-1)g_n)\frac{P_{n-1}^{(r-1)}(1)}{P_{n}(1)}\\
&=&f_n\frac{P_{n-1}(1)}{P_{n}(1)}\left(1+(r-1)\frac{g_n}{f_n}\right)
\frac{P_{n-1}^{(r-1)}(1)}{P_{n-1}(1)}.
\eeq*
Therefore, upon further iteration,
\beq*
\frac{P_n^{(r)}(1)}{P_n(1)}&=&\left(\prod_{k=0}^{r-1}f_{n-k}{\frac{P_{n-k-1}(1)}{P_{n-k}(1)}\left(1+(r-k-1)\frac{g_{n-k}}{f_{n-k}}\right)}\right)\frac{P_{n-r}^{(r-r)}(1)}{P_{n-r}(1)}.
\eeq*
Since the last factor is \begin{math}1\end{math}, it follows from (\ref{pol_ass}) that for every \begin{math}r\ge 1\end{math}  as 
 \begin{math}n\rightarrow \infty\end{math},
\beq*
\frac{P_n^{(r)}(1)}{P_n(1)}&\rightarrow&c^r
\eeq*
as desired.
\end{proof}

\begin{remark} In principle it should  be possible to prove a similar result  for polynomials of higher degrees than those considered in Proposition~\ref{prop:pol_recur}. However, we have not tried to do that, primarily because we have not encountered instances of such recurrences.
\end{remark}

\section{Real--rootedness of $P_n(x)$}
The idea we explore in this section is that if all roots of \begin{math}P_n(x)\end{math}  are real then \begin{math}P_n(x)\end{math}  can be written as a product of linear factors. Furthermore, since the coefficients are non--negative the roots are non--positive. Hence, these linear factors may be interpreted as the generating functions of \begin{math}\{0,1\}\end{math}--valued random variables and then knowing that the variance of their sum converges to infinity suffices to conclude that the sum is asymptotically normal. More specifically, assume that  
\[-\infty<\gamma_{i,n}\le0, i=1,\dots,m\]
are roots of \begin{math}P_n(x)\end{math}  and write \begin{math}\pi_{i,n}=-\gamma_{i,n}\end{math}  so that \begin{math}\pi_{i,n}\ge0\end{math}.
Then \begin{math}P_n(x)\end{math}  has a factorization 
\[P_n(x)=p_{n,m}\prod_{k=1}^m(x+\pi_{k,n}),\]
so that 
\[\E x^{X_n}=\frac{P_n(x)}{P_n(1)}=\prod_{k=1}^m\frac{x+\pi_{k,n}}{1+\pi_{k,n}}=\prod_{k=1}^m\left(\frac x{1+\pi_{k,n}}+\frac{\pi_{k,n}}{1+\pi_{k,n}}\right).
\] 
The factor on the right--hand side is the probability generating function of a random variable \begin{math}\xi_{k,n}\end{math}  such that 
\[\P(\xi_{k,n}=1)=\frac1{1+\pi_{k,n}} \quad\mbox{and}\quad\P(\xi_{k,n}=0)=\frac{\pi_{k,n}}{1+\pi_{k,n}},\quad k=1,\dots,m.\]
Moreover, since the product of the probability generating functions corresponds to taking sums of independent random variables 
we have that 
\[X_n=\sum_{k=1}^n\xi_{k,n},\]
where \begin{math}(\xi_{k,n})\end{math}  are independent. Therefore, it follows immediately from either Lindeberg or Lyapunov version of the central limit theorem (see e.~g. \cite[Theorem~27.2 or Theorem~27.3]{Bil}) that 
\[\frac{X_n-\E X_n}{\sqrt{\textrm{var}(X_n)}}\stackrel d\longrightarrow N(0,1),\]
as long as \begin{math}\textrm{var}(X_n)\longrightarrow\infty\end{math}  as \begin{math}n\to\infty\end{math}. (Here \begin{math}N(0,1)\end{math}  denotes the standard normal random variable.)

Since showing that the variance of \begin{math}X_n\end{math}  tends to infinity is generally not difficult from the recurrences (\ref{rec_prime}) and (\ref{rec}), the main issue is real--rootedness of \begin{math}P_n(x)\end{math}. This is, of course, not a new idea and the problem has a very long history and the questions of real--rootedness for many families of classical polynomials have been settled long time ago.  In particular, in the context the present discussion,  the proof that all roots of polynomials {\bf(HJ)} are real was a slight modification of the proof  for the Eulerian polynomials given by~\cite{F} more than hundred years ago. Nonetheless, the techniques seem to be tailored to the particular cases at hand. As far as general criteria for the real--rootedness of a family of recursively defined polynomials, not much seem to have been known until two relatively recent papers \cite{DDJ,LW}.    The first  concerns recurrence (\ref{rec}) and requires  \begin{math}f_n(x)\end{math}  and \begin{math}g_n(x)\end{math}  to have degrees at most one and two, respectively. The second, when specified to generality of (\ref{rec}) does not put any restrictions on the degrees of \begin{math}f_n(x)\end{math}  and \begin{math}g_n(x)\end{math}  but requires that \begin{math}g_n(x)<0\end{math}  whenever \begin{math}x\le0\end{math}. While many of the real--rootedness  results for classical polynomials may obtained from one of  these criteria (and sometimes from both, e.~g.~Eulerian or Bell polynomials) some  are not covered by them. In particular, neither \cite{DDJ} nor \cite{LW} applies to our first example {\bf(ABN)}. Yet, as it turns out a modification of methods developed in \cite{DDJ} may be used to show that the polynomials \begin{math}B_n(x)\end{math}  defined by \textbf{(ABN)} do, indeed,  have all roots real.  We will not prove it in this extended abstract, instead referring the reader to the full version of this paper.

\section{Asymptotic normality of the number of diagonal boxes in symmetric tree--like tableaux}
In this section we analyze the recurrence {\bf(ABN)}. The  polynomials 
\[B_{n}(x)=\sum_{k=1}^{n+1}B(n,k)x^k,\quad n\ge0,\] 
were introduced in \cite[Section~3.2]{ABN} and are the generating polynomials for the number of diagonal cells in symmetric tree--like tableaux of size \begin{math}2n+1\end{math}   (that is to say that \begin{math}B(n,k)\end{math}  is the number of symmetric tree--like tableaux of size \begin{math}2n+1\end{math}  with \begin{math}k\end{math}  diagonal cells). As was shown  in \cite{ABN}  \begin{math}(B_n(x))\end{math}  satisfy the recurrence \textbf{(ABN)} and it follows readily from that that the expected number of diagonal cells  in  symmetric tableaux of size \begin{math}2n+1\end{math}  is \begin{math}3(n+1)/4\end{math}  (see \cite[Proposition~19]{ABN}). Continuing that work, we find the expression for the variance and show that the number 
of diagonal cells is asymptotically normal. The precise statement is as follows. 
\begin{theorem}\label{thm:clt} Let \begin{math}D_n\end{math}  be the number of diagonal boxes in a random symmetric tree--like tableau of size \begin{math}2n+1\end{math}. Then, as \begin{math}n\to\infty\end{math} 
\[\frac{D_n-3(n+1)/4}{\sqrt{7(n+1)/48}}\stackrel d\longrightarrow N(0,1).\]
\end{theorem}
Since \begin{math}(D_n)\end{math}  are random variables defined by 
\[\P(D_n=k)=\frac{B(n,k)}{\sum_{k\geq0}B(n,k)}=\frac{B(n,k)}{B_{n}(1)},\]
where \begin{math}(B_n(x))\end{math}  satisfy recurrence \textbf{(ABN)} it follows form our discussion that theorem will be proved once we show that the variance of \begin{math}D_n\end{math}  grows to infinity with
 \begin{math}n\end{math}  and that all roots of \begin{math}B_n(x)\end{math}  are real.  The precise statements are given it two propositions below.
\begin{proposition}\label{prop:var}
The variance of the number of diagonal cells in a random symmetric tree--like tableaux of size \begin{math}2n+1\end{math}  is,
\be\label{var_expl}
\textrm{var}(D_n)=\frac{7(n+1)}{48}.
\ee
\end{proposition}

\begin{proposition}\label{prop:roots}
For all \begin{math}n\geq0\end{math}, the polynomial \begin{math}B_{n}(x)\end{math} 
\begin{itemize}
\item[a)]{} has degree \begin{math}n+1\end{math}  with all coefficients non-negative, and
\item[b)]{} all roots real and in the interval \begin{math} [-1,0]\end{math}.
\end{itemize}
\end{proposition}
Because of the space limitation we will include here a proof of Proposition~\ref{prop:var} only and we refer the reader to the full version of the paper for the proof of Proposition~\ref{prop:roots}.  

\begin{proof}[
 of Proposition~\ref{prop:var}]
First we will calculate the second factorial moment of \begin{math}D_{n}\end{math}. 
Differentiating the recurrence \textbf{(ABN)} twice and evaluating at \begin{math}x=1\end{math}  yields
\[B_n''(1)=2nB_{n-1}(1)+6(n-1)B_{n-1}'(1)+2(n-2)B_{n-1}''(1).\]
Furthermore, since
\[B_n(1)=2nB_{n-1}(1)\]
and 
\be\label{var}\textrm{var}(D_{n})=\E(D_{n})_2-\E^2D_{n}+\E D_{n}
\ee
we obtain
\beq*
\E(D_{n})_2&=&\frac{B^{''}_{n}(1)}{B_{n}(1)}
=\frac{2nB_{n-1}(1)+6(n-1)B_{n-1}'(1)+2(n-2)B_{n-1}''(1)}
{2nB_{n-1}(1)}\\
&=&1+\frac{3(n-1)}{n}\E D_{n-1}+\frac{n-2}{n}\E(D_{n-1})_2
\\
&=&1+\frac{3(n-1)}{n}\E D_{n-1}+\frac{n-2}{n}\left(\textrm{var}(D_{n-1})+\E^2D_{n-1}-\E D_{n-1}\right)
\\
&=&1+\frac{n-2}{n}\textrm{var}(D_{n-1})+\frac{n-2}{n}\E^2D_{n-1}+\left(\frac{2n-1}{n}\right)\E D_{n-1}.
\eeq*
Now, using $\E D_{n}=3(n+1)/4$ (as computed from {\bf(ABN)} in \cite[Proposition~19]{ABN}) and (\ref{var}) we obtain
\beq*
\textrm{var}(D_{n})&=&1+\frac{n-2}{n}\textrm{var}(D_{n-1})+\frac{n-2}{n}\left(\frac{3n}4\right)^2+\frac{2n-1}{n}\frac{3n}4\\&&\qquad-\left(\frac{3(n+1)}{4}\right)^2+\frac{3(n+1)}{4}\\
&=&\frac{n-2}{n}\textrm{var}(D_{n-1})+\frac{7}{16}.
\eeq*
This recurrence is easily solved (see e.~g. \cite[Section~2.2]{GKP}) and yields (\ref{var_expl}) completing the proof of Proposition~\ref{prop:var} and Theorem~\ref{thm:clt}.
\end{proof}

\begin{remark} The representation of \begin{math}D_n\end{math}  as the sum of independent indicator random variables implies that a  local limit theorem holds too. Specifically,
using 
\begin{math}\E D_n=3(n+1)/4\end{math}  and \begin{math}\textrm{var}(D_n)=7(n+1)/48\end{math}  
we have that
\[\P(D_n=k)=\frac{2\sqrt6}{\sqrt{7\pi(n+1)}}\left(\exp\left(-\frac{24(k-3(n+1)/4)^2}{7(n+1)}\right)+o(1)\right)\]
holds uniformly over \begin{math}k\end{math}  as \begin{math}n\to\infty\end{math}.
We refer to \cite[Theorem~2.7 and a discussion of its proof in Section~5]{HJ}  for more detailed explanation and to \cite[Theorem~VII.3]{P} for a general statement of a local limit theorem.
\end{remark}

\section{Conclusion}

We have considered recurrences for generating polynomials of sequences of integer valued random variables  and tried to use these recurrences to identify  the distributional limits of the associated sequences of random variables. Some examples lead to Poisson limits, some other to Gaussian limits.  In particular, we established the asymptotic normality for the number of diagonal cells in the random tree--like tableaux by verifying that the generating polynomials have only real roots and that the variance tends to infinity with $n$. However, there seem to be lack of general criteria that would allow one to find the limiting distribution of the underlying sequence of random variables directly from the recurrences of the form (\ref{rec}) or (\ref{rec_prime}). For example, the limiting distribution of the random variables associated with the recurrence  {\bf(AH)} is neither Poisson nor normal. In fact, as have been shown in \cite[Section~3]{AH} if \begin{math}(X_n)\end{math}  is a sequence of random variables associated with the recurrence {\bf(AH)} through (\ref{rv}) then 
\[\frac{X_n}{2\sqrt n}\stackrel d\longrightarrow X,\]
where \begin{math}X\end{math}  is a random variable with the probability density function \begin{math}2xe^{-x^2}\end{math}  if \begin{math}x\ge0\end{math}  and is 0 otherwise. However, it is not clear how to see it 
from the recurrence {\bf(AH)}. Factorial moments
satisfy 
\begin{eqnarray*}\E{(X_{n})_r}&=&
\frac{2n-1+r}{2n-1}\E{(X_{n-1})_r}+\frac{r(r-1)}{2n-1}\E{(X_{n-1})_{r-1}}
\end{eqnarray*}
and one can get from there
\[
\E{X_{n}}=\frac{2n}{2n-1}\E{X_{n-1}} =
\frac{(2n)!!}{(2n-1)!!}=\frac{2^{2n}}{{2n\choose n}}\sim\sqrt{\pi n} 
\]
and
\[
\textrm{var}(X_{n})= (4-\pi)n+O(\sqrt n).\]
In principle, higher moments can be found too. For example 
\[
\E{(X_{n})_3}=6\left(\frac{\sqrt\pi(n+2)n!}{\Gamma(n+1/2)}-4n-3\right)\sim6\sqrt\pi n^{3/2}
\]
but the computations become increasingly more complicated. Even the asymptotic behavior of the first two moments is not immediately clear from the recurrence {\bf(AH)}.

Thus,  it seems worthwhile to further study the recurrences like  (\ref{rec_prime}) and (\ref{rec}) to obtain a more comprehensive picture of their probabilistic consequences.

\bibliographystyle{abbrvnat}
%

%
%

\bibliography{poly}
\label{sec:biblio}

\end{document}